\documentclass[11pt]{article}
 \usepackage{graphicx}
 \usepackage{amssymb}
 \usepackage{epstopdf}
 \usepackage{pifont}
 \usepackage{amsmath}
 \usepackage{enumerate}
 \usepackage{graphicx}
 \usepackage{stmaryrd}
 \usepackage{amsthm}
 \usepackage{verbatim}
 \usepackage[all]{xypic}
 \usepackage[mathscr]{euscript}
 \usepackage{mystyle}
 \usepackage{mathrsfs}

 \usepackage{ccfonts}

 \newtheoremstyle{mythmsty}{15pt}{5pt}{\it}{0pt}{\textbf}{}{0pt}{\textbf{(\thmnumber{#2})} \thmname{#1} \thmnote{ #3}
}

 \theoremstyle{mythmsty}

 \newtheorem{thm}[paragraph]{\bf Theorem}
 \newtheorem{cor}[paragraph]{\bf Corollary}
 
 \newtheorem{prop}[paragraph]{\bf Proposition}
 \newtheorem{lemma}[paragraph]{\bf Lemma}

 \newenvironment{thm*}{\bf Theorem \it}{}
 \newenvironment{cor*}{\bf Corollary \it}{}
 \newenvironment{definition*}{\bf Definition \it}{}
 \newenvironment{prop*}{\bf Proposition \it}{}
 \newenvironment{lemma*}{\bf Lemma \it}{}
 \newenvironment{fact*}{\bf Fact \it}{}
 \newenvironment{rmk*}{\bf Remark \it}{}
 \newenvironment{exercise*}{\bf Exercise \it}{}
 \newenvironment{solution*}{\underline{Solution}\quad\rm}

  \numberwithin{equation}{subsubsection}

\newcommand{\hhh}[1]{\vskip .5in ******************* HERE! ********************\\ \vskip .1in $\Rightarrow$#1$\Leftarrow$}
\newcommand{\mysec}[1]{\noindent \section{#1} \label{#1}}
\newcommand{\mysubsec}[1]{\noindent \subsection{#1} \label{#1}}
\newcommand{\ppp}[1]{\noindent\subsubsection{}\label{#1}\hskip -.1455in\textbf{)}}
\newcommand{\rf}[1]{(\ref{#1})}

\title{\bf  Functors on triangulated tensor categories}
\author{{\small \textsc{ Yu-Han Liu}}}
\begin{document}
\date{}
\maketitle

\setcounter{tocdepth}{2}

\setcounter{section}{0}

\mysec{Introduction}

\mysubsec{Introduction}

\ppp{}  In this paper we define the functorial spectrum $|\mathrm{Sp}(T)|$ for every triangulated tensor category; this means that $T$ is a triangulated category on which we have a monoidal category structure, or a ``tensor product'', $(a,b)\mapsto a\otimes b$ which is exact in both variables.  The central idea of replacing the category $T$ with the functor $\mathrm{Sp}(T)$ it represents had already appeared in \cite{liu_sierra_quiver}.  

Roughly speaking, elements in $|\mathrm{Sp}(T)|$ are (isomorphism classes of) functors from $T$ into $D^b(\mathrm{Vect}_k)$ which preserve the unit object and the tensor product.  This construction may be viewed as analogous to the development of algebraic geometry from a functorial point of view, as done in \cite{demazure_gabriel_intro}.

In general $|\mathrm{Sp}(T)|$ comes with some natural structures:  It is always a locally ringed space \rf{Structure sheaf}, and natural transformations between functors gives a ``path algebra'' structure \rf{Path algebra}.  The construction $T\mapsto |\mathrm{Sp}(T)|$ is contravariant in $T$.

\ppp{}  Our construction is closely related to the prime spectrum $\mathrm{Spec}(T)$ defined by Balmer \cite{balmer_spectrumtt}.  Indeed there is always a comparison morphism \[|\mathrm{Sp}(T)|\longrightarrow \mathrm{Spec}(T).\]  This morphism is not always an isomorphism \rf{par:ex}.

Balmer \cite{balmer_spectrumtt} showed that in the case when $X$ is a topologically noetherian scheme and $T_X:=D(X)_\mathrm{parf}$ the triangulated tensor category of perfect complexes \cite[3.1]{thomason_class} on $X$ (along with the derived tensor product), we have a natural comparison isomorphism \[X\longrightarrow \mathrm{Spec}(T_X).\]  

We will show in \rf{thm:main} that under the same assumptions we have comparison isomorphisms

\[X\longrightarrow |\mathrm{Sp}(T_X)|\longrightarrow \mathrm{Spec}(T_X).\]

\ppp{}  On the other hand, the path algebra structure allows us to recover finite ordered quivers from their representations \cite{liu_sierra_quiver}; see \rf{thm:quiver} below.

\ppp{}  From the point of view of the reconstruction result \rf{thm:main}, we see that non-isomorphic schemes $X,X'$ with equivalent categories $T_X\cong T_{X'}$ give distinct tensor products on the same triangulated category.  Conversely, starting with a scheme $X$, in order to understand other schemes $X'$ having equivalent category of perfect complexes, we need to understand possible tensor products on the triangulated category $T_X$.

Following this philosophy we give an alternative proof of the reconstruction theorem \rf{thm:BO} due to Bondal-Orlov \cite{BO_reconstruction}.  The proof in \cite{BO_reconstruction} may be interpreted as saying that a smooth projective variety $X$ with ample (anti-)canonical bundle is the moduli space of ``point objects'' in $T_X=D^b(X)$.  This idea was later used to construct flops by Bridgelend in \cite{bridgeland_flops}, where the meaning of ``point objects'' depends on the choice of a $t$-structure.

From our point of view, however, the space of $X$ is constructed as the functorial spectrum $|\mathrm{Sp}(T_X)|$ first.  The reconstruction theorem by Bondal-Orlov is then reinterpreted as the uniqueness of tensor product on the triangulated category $T_X$.  

\ppp{}  When we work in the derived category $D(\mathcal O_X\mathrm{-mod})$ all functors considered are derived functors.  For example, if $f:Y\rightarrow X$ is a morphism between schemes we denote simply by $f^*$ and $f_*$ the derived functors $Lf^*$ and $Rf_*$, respectively. 

\mysec{Representable functors on triangulated tensor categories}

\mysubsec{The set of points}

\ppp{par:TT}  Let $\mathbf{TT}$ be the category of essentially small triangulated categories which is also a \emph{strict} monoidal category \cite[Chapter VII, section 1.]{maclane_categories}, satisfying the following compatibility condition with the triangulated category structure:  The tensor product is required to be exact in each variables.

 Morphisms in $\mathbf{TT}$ are exact functors which are also \emph{strong} monoidal functors, that is, the unit objects and tensor products are preserved up to natural \emph{isomorphisms} \cite[Chapter VII, section 1.]{maclane_categories}.  Morphisms in $\mathbf{TT}$ will be simply called \emph{tensor functors}. 
 
If $g: T\rightarrow T'$ is a tensor functor which is also an equivalence of categories, then by \cite[Chapter VI, section 4]{maclane_categories} and \cite[Lemma 1.2]{orlov_k3} we know that its quasi-inverse $g': T'\rightarrow T$ is exact, and can be shown to be also a strong monoidal functor.

For any $T\in\mathbf{TT}$ denote by $\mathrm{Sp}(T)$ the set-valued covariant functor $\mathrm{Hom}_{\mathbf{TT}}(T,-)/\cong$ on $\mathbf{TT}$, sending $S\in\mathbf{TT}$ to the natural isomorphism classes of tensor functors from $T$ to $S$.  We will typically denote elements in $\mathrm{Sp}(T)(S)$ simply by representative functors $F:T\rightarrow S$.

\ppp{par:field}  For any field $k$ denote by $\mathrm{Vect}_k$ the abelian category of finite dimensional $k$-vector spaces, and by $T_k$ the derived category $D^b(\mathrm{Vect}_k)$ with the usual tensor product and unit object $k$.  We denote $\mathrm{Sp}(T)(T_k)$ simply by $\mathrm{Sp}(T)(k)$.  

The category $T_k$ is equivalent to the subcategory $\displaystyle\bigoplus_j \mathrm{Vect}_k[j]$ by taking cohomology $V\mapsto H^*(V)$.  Note that taking cohomology is a tensor functor; we will often identify $T_k$ with $\displaystyle\bigoplus_j \mathrm{Vect}_k[j]$.


\ppp{par:def}  For any field extension $k\rightarrow k'$ we have a tensor functor $-\otimes_kk': T_k\rightarrow T_{k'}$ and hence a map $\mathrm{Sp}(T)(k)\rightarrow \mathrm{Sp}(T)(k')$.  Let \[|\mathrm{Sp}(T)|:=\varinjlim_{k} \mathrm{Sp}(T)(k).\]  More precisely, $|\mathrm{Sp}(T)|$ is the disjoint union of $\mathrm{Sp}(T)(k)$ modulo the equivalence relation generated by $F\sim F'$ if $F\in \mathrm{Sp}(T)(k)$ is mapped to $F'\in \mathrm{Sp}(T)(k')$ for some field extension $k\rightarrow k'$.  For example, $|\mathrm{Sp}(T_k)|$ consists of one point for every field $k$.  We denote by $[F]$ the element in $|\mathrm{Sp}(T)|$ represented by a functor $F: T\rightarrow T_k$.


In the following sections we introduce some natural structures on the set $|\mathrm{Sp}(T)|$.

\ppp{par:subfunctors}  Notice that the construction above makes sense for any set-valued covariant functor $\mathscr X$ on the category $\mathbf{TT}$, giving a set $|\mathscr X|$.  This way we get a set-valued covariant functor on the category $\mathrm{Fun}(\mathbf{TT},\mathbf{Set})$ of functors on $\mathbf{TT}$.  

We say a subfunctor $\mathscr U\subset\mathscr X$ is \emph{open} if for every field extension $k\rightarrow k'$ inducing $f:\mathscr X(k)\rightarrow \mathscr X(k')$ we have $f^{-1}\mathscr U(k')=\mathscr U(k)$; note that the containment $\supset$ holds for any subfunctor of $\mathscr X$.  

This property allows us to ``fold'' the equivalence relation:  If $\mathscr U\subset \mathscr X$ is open then a point $[u]\in|\mathscr X|$ lies in the image of the natural map $|\mathscr U|\rightarrow |\mathscr X|$ if and only if $u\in \mathscr U(k)$ for \emph{every} representative $u$ of $[u]$.

In particular in this case $|\mathscr U|$ can be naturally identified with a subset of $|\mathscr X|$.  It is then straightforward to show that \[\left|\bigcup_i\mathscr U_i\right|=\bigcup_i |\mathscr U_i|,\]and \[|\mathscr U_1\cap \mathscr U_2|=|\mathscr U_1|\cap|\mathscr U_2|,\]as subsets of $|\mathscr X|$ for any collection $\mathscr U_i$ of \emph{open} subfunctors of $\mathscr X$.

\mysubsec{Topology}

\ppp{}  For any collection $M$ of objects in $T$, denote by $\mathscr U_M$ the subfunctor of $\mathrm{Sp}(T)$ defined by \[\mathscr U_M(S):=\{F:T\rightarrow S\,|\,0\in F(M^\otimes)\},\]where $M^\otimes$ denotes the the collection of finite tensor products of objects in $M$.  Notice that we have \[\mathscr U_M=\bigcup_{a\in M}\mathscr U_a,\]and \[\mathscr U_{M_1}\cap \mathscr U_{M_2}=\mathscr U_{M_1\oplus M_2},\]where $M_1\oplus M_2$ denotes the collections of objects $a_1\oplus a_2$ with $a_i\in M_i$.

\ppp{}  Since the functor $-\otimes_kk':T_k\rightarrow T_{k'}$ for every field extension $k\rightarrow k'$ is injective on objects, the subfunctor $\mathscr U_M$ is open the sense of \rf{par:subfunctors}.

A subset of $|\mathrm{Sp}(T)|$ is called \emph{open} if it is of the form $U_M:=|\mathscr U_M|$ for some $M\subset T$; by \rf{par:subfunctors} we see that this indeed defines a topology on $|\mathrm{Sp}(T)|$.  

Notice that if $M$ is a finite set, let $a$ be the tensor product of objects in $M$, then we have \[\mathscr U_M=\mathscr U_a.\]  Hence every quasi-compact open subset of $|\mathrm{Sp}(T)|$ is of the form $U_a=|\mathscr U_a|$ for some object $a\in T$.

\ppp{par:supportdata}  As an example, for $a\in T$ the open set $U_a$ consists of classes $[F]\in|\mathrm{Sp}(T)|$ of functors $F: T\rightarrow T_k$ with $F(a)=0$.  Define the \emph{support} $\mathrm{s}(a)$ of $a$ to be the complement of $U_a$, then the association \[a\mapsto \mathrm{s}(a)\] gives a \emph{support data} $(|\mathrm{Sp}(T)|,\mathrm{s}(-))$ on $T$ in the sense of \cite[Definition 3.1]{balmer_presheaves}.

\mysubsec{Structure sheaf}

\ppp{par:R}  For any open set $U\subset |\mathrm{Sp}(T)|$ let \[T^U=\bigcap_{[F]\in U}\ker(F);\]note that this is well-defined since for any field extension $k\rightarrow k'$ we have that $\ker(-\otimes_kk')$ consists of objects isomorphic to zero, hence the full triangulated subcategory $\ker(F)$ is independent of the choice of representative in the class $[F]\in U$.  Notice that $T^U$ is a thick tensor ideal, and the localization $T/T^U$ is again a triangulated tensor category.  

More explicitly, with the notation in \rf{par:supportdata} we see that $T^U$ consists of the objects $a\in T$ such that $U$ is contained in $U_a$; or equivalently, the support $\mathrm{s}(a)$ is contained in the closed subset $|\mathrm{Sp}(T)|-U$. 

Define \[R(U):=\mathrm{End}_{T/T^U}(1),\]where $1$ denotes the image of the unit object in $T/T^U$.

Notice that the construction $U\mapsto T^U$ reverses inclusion, hence $U\mapsto R(U)$ is a presheaf of commutative rings \cite[Lemma 9.6]{balmer_spectrumtt}.  Denote by $\mathcal O_T$ its sheafification, then the pair $(|\mathrm{Sp}(T)|,\mathcal O_T)$ is a ringed space.

\ppp{}  We summarize some simple observations to be used freely below:

\begin{lemma}\label{lem:basic}  Let $M,M'$ be collections of objects in $T$, $a,a'$ objects in $T$, $U,U'$ open subsets of $|\mathrm{Sp}(T)|$, and $[F]\in |\mathrm{Sp}(T)|$.  The following statements hold:
\begin{enumerate}[(1)]
\item $M\subset M'\Rightarrow U_M\subset U_{M'}$.
\item $U\subset U'\Rightarrow T^U\supset T^{U'}$.
\item $[F]\in U_M \Leftrightarrow \ker(F)\cap M\neq\emptyset$.
\item $[F]\in U \Rightarrow \ker(F)\supset T^U$.
\item $[F]\in U_a\Leftrightarrow a\in \ker(F)\Rightarrow \ker(F)\supset T^{U_a}$.
\item $a\in T^{U_a}\subset T^{U_{a\oplus a'}}$.
\item $U_{a\oplus a'}\subset U_a$.

\end{enumerate}
\end{lemma}

\ppp{par:stalk}  We now compute the stalk $\mathcal O_{T,[F]}$ at a point $[F]\in|\mathrm{Sp}(T)|$.  By definition we have \[\mathcal O_{T,[F]}=\varinjlim_{[F]\in U}\mathrm{End}_{T/T^U}(1).\]

Recall that every open set $U$ is of the form $\displaystyle U_M=\bigcup_{a\in M}U_a$ for some $M\subset T$.  Hence if $[F]\in U$ then $[F]\in U_a$ for some $a$, which is equivalent to the condition that $a\in \ker(F)$.  In other words, we have \[\mathcal O_{T,[F]}=\varinjlim_{a\in \ker(F)}\mathrm{End}_{T/T^{U_a}}(1).\]

Notice that for every $a\in \ker(F)$ we have a factorization 

\centerline{
\xymatrix{T\ar[rrr]^-F\ar[dr] &&& T_k \\ & T/T^{U_a} \ar[r] & T/\ker(F). \ar[ur]_-{\bar F}}
}
The homomorphisms $\mathrm{End}_{T/T^{U_a}}(1)\rightarrow \mathrm{End}_{T/\ker(F)}(1)$ then induces a homomorphisms \[\phi: \mathcal O_{T,[F]}\longrightarrow \mathrm{End}_{T/\ker(F)}(1).\]  

\begin{prop}\label{prop:stalk}  The homomorphism $\phi$ above is an isomorphism.  Moreover, $\mathrm{End}_{T/\ker(F)}(1)$ is a local ring; in particular $(|\mathrm{Sp}(T)|,\mathcal O_T)$ is a locally ringed space.
\end{prop}

\begin{proof}  We will show more generally that the natural map \[\phi: \varinjlim_{a\in\ker(F)}\mathrm{Hom}_{T/T^{U_a}}(x,y)\longrightarrow \mathrm{Hom}_{T/\ker(F)}(x,y)\] is an isomorphism for any $x,y\in T$, viewed as objects in the localizations.  

We first show that it is surjective:  Any element $f\in \mathrm{Hom}_{T/\ker(F)}(x,y)$ is of the form \[x\stackrel{s}{\longleftarrow} w \stackrel{g}{\longrightarrow}y,\] where $s$ and $g$ are morphisms in $T$ so that $F(s)$ is an isomorphism.  This means that $F(\mathrm{cone}(s))=0$; let $a:=\mathrm{cone}(s)\in\ker(F)$, then we see that $f$ is a morphism in $T/T^{U_a}$ since $a\in T^{U_a}$, hence the element $[f]$ in $\displaystyle\varinjlim_{a\in\ker(F)}\mathrm{Hom}_{T/T^{U_a}}(x,y)$ represented by $f\in\mathrm{Hom}_{T/T^{U_a}}(x,y)$ maps to $f$ under $\phi$.
  
We now show that $\phi$ is injective:  Let $[f]$ be in the kernel of $\phi$.  By definition $[f]$ is represented by some $f\in \mathrm{Hom}_{T/T^{U_a}}(x,y)$, for some $a\in \ker(F)$.  Hence $f$ is of the form $x\stackrel{s}{\leftarrow}w\stackrel{g}{\rightarrow}y$ with $s,g$ in $T$, and $s$ is an isomorphism in $T/T^{U_a}$.  Since $s$ is also an isomorphism in $T/\ker(F)$ we see that $g$ is mapped to zero in $T/\ker(F)$.

Consider distinguished triangle in $T$: \[w\stackrel{g}{\longrightarrow} y\stackrel{h}{\longrightarrow} \mathrm{cone}(g).\]  We see that $h$ has a left inverse in $T/\ker(F)$.  By the surjectivity above, we can find a morphism $m\in \mathrm{Hom}_{T/T^{U_b}}(\mathrm{cone}(g),y)$ that maps to this left inverse in $T/\ker(F)$.  Since $T^{U_a}\subset T^{U_{a\oplus b}}$, both $f$ and $m$ may be viewed as a morphism in $T/T^{U_{a\oplus b}}$.  

Now the composition $m\circ h\circ g$ is equal to zero in $T/T^{U_b}$, but $m\circ h$ is an isomorphism in $T/\ker(F)$.  This means that $c:=\mathrm{cone}(m\circ h)$ lies in $\ker(F)$, and moreover $m\circ h$ is already an isomorphism in $T/T^{U_c}$.  Let $d=a\oplus b\oplus c$, then in $T/T^{U_d}$ we have $m\circ h\circ g=0$ but $m\circ h$ is an isomorphism, hence $g$ is zero in $T/T^{U_d}$, that is, $[f]=0$ as required for the injectivity of $\phi$.

Finally we show that $\mathrm{End}_{T/\ker(F)}(1)$ is a local ring.  Denote by $\bar F_1$ the homomorphism \[\mathrm{End}_{T/\ker(F)}(1)\longrightarrow \mathrm{End}_{T_k}(F(1))\cong k;\]recall that $F(1)\cong k$ placed at degree zero in $T_k=D^b(\mathrm{Vect}_k)$.  It suffices to show that every $f\in \mathrm{End}_{T/\ker(F)}(1)$ not lying in $\ker(\bar F_1)$ is invertible, but this is just the definition of localization $T/\ker(F)$:  Since $k$ is a \emph{field}, $f\notin \ker(\bar F_1)$ implies that $\bar F_1(f)$ is invertible in $T_k$, in other words an isomophism.  Hence $f$ is already an isomorphism in $T/\ker(F)$.  \end{proof}



\mysubsec{Path algebra}

\ppp{}  For any field $k$ and $F,G\in \mathrm{Sp}(T)(k)$ define $\mathrm{Hom}(F,G)$ to be the $k$-vector space spanned by natural transformations from $F$ to $G$.  Let $A(T,k)$ be the direct sum of all $\mathrm{Hom}(F,G)$ with $F,G\in \mathrm{Sp}(T)(k)$, then by composing natural transformations and $k$-linearity $A(T,k)$ is a $k$-algebra.

\ppp{par:path}  With $T$ fixed the association $k\mapsto A(T,k)$ is functorial in the sense that for every field extension $k\rightarrow k'$ we have an algebra homomorphism \[\psi_{k\rightarrow k'}: A(T,k)\otimes_kk'\longrightarrow A(T,k').\]

This way $A(T,-)$ is an algebra-valued functor on the category of fields; we call this functor the \emph{path algebra on $|\mathrm{Sp}(T)|$}.  Recall that we have maps $\mathrm{Sp}(T)(k)\rightarrow |\mathrm{Sp}(T)|$ for every $k$ sending $F\mapsto [F]$.  The algebra $A(T,k)$ may be considered as giving a quiver with $\mathrm{Sp}(T)(k)$ as the set of vertices.

\mysubsec{Functoriality}

\ppp{}  Any morphism $g:T\rightarrow T'$ in $\mathbf{TT}$ induces by composition a natural transformation $\mathrm{Sp}(T')\rightarrow \mathrm{Sp}(T)$ between set-valued functors.  This in turn induced a map \[\gamma:|\mathrm{Sp}(T')|\longrightarrow |\mathrm{Sp}(T)|.\]

In the following we show that this induced map preserves the various structures introduced above.

\ppp{}  Let $M\subset T$ be a collection of objects, and $U_M=|\mathscr U_M|$ the associated open subset of $|\mathrm{Sp}(T)|$.  Let $M'=g(M)$, then it is straightforward to check that that $U_{M'}=\gamma^{-1}(U_M)$, hence $\gamma$ is continuous.


\ppp{}  Let $U\subset |\mathrm{Sp}(T)|$ be an open subset, and $U'=\gamma^{-1}(U)$.  Recall that an object $a\in T$ lies in $T^U$ if and only if $U\subset U_a$.  This last containment implies $U'\subset \gamma^{-1}(U_a)=U_{g(a)}$, hence $g$ gives a functor from $T^U$ into $T'^{U'}$.  Hence we have an induced functor \[\bar g_U:T/T^U\longrightarrow T'/T'^{U'}.\]

Recall the presheaf $R:U\mapsto \mathrm{End}_{T/T^U}(1)$ on $|\mathrm{Sp}(T)|$ whose sheafification was defined to be $\mathcal O_T$; similarly we have a presheaf $R'$ on $|\mathrm{Sp}(T')|$.  Then $U\mapsto \bar g_U$ gives a presheaf morphism \[R\longrightarrow \gamma_*R',\]which in turns gives \[\gamma^\#:\mathcal O_T\longrightarrow \gamma_*\mathcal O_{T'},\]making $(\gamma,\gamma^\#)$ a morphism between ringed spaces.

\ppp{}  Given $[F']\in |\mathrm{Sp}(T')|$ we have $\gamma([F'])=[F'\circ g]=:[F]\in |\mathrm{Sp}(T)|$.  The induced map \[\gamma^\#_{[F']}: \mathcal O_{T,[F]}\longrightarrow \mathcal O_{T',[F']}\] on stalks is \[\mathrm{End}_{T/\ker(F)}(1)\longrightarrow \mathrm{End}_{T'/\ker(F')}(1),\]which is induced by the restriction of $g$ from $\ker(F)=\ker(F'\circ g)$ to $\ker(F')$.  We know that these are local rings with maximal ideals respectively $\ker(\bar F(1))$ and $\ker(\bar F'(1))$.  Clearly $\gamma^\#_{[F']}$ maps the first ideal into the second, hence $(\gamma,\gamma^\#)$ is a morphism of locally ringed spaces.

\ppp{}  If $\phi: F\rightarrow G$ is a natural transformation with $F,G\in \mathrm{Sp}(T')(k)$, then for any $a\in T$, \[a\mapsto\phi_{g(a)}: F(g(a))\longrightarrow G(g(a))\] is a natural transformation from $F\circ g$ to $G\circ g$.  This extends to a natural transformation \[A(g): A(T',-)\longrightarrow A(T,-)\] between path algebras.

\mysubsec{Examples}

\ppp{par:eqisom}  If $g: T\rightarrow T'$ is tensor functor which is also an equivalence of categories, then the induced morphism $\gamma$ is an isomorphism; see \rf{par:TT}.

\ppp{}  Consider the localization functor $T\rightarrow T/T^{U_a}$ with $a\in T$.  Recall that $a\in T^{U_a}$ by \rf{lem:basic}, hence we have a natural injection \[\mathrm{Sp}(T/T^{U_a})(k) \stackrel{}{\longrightarrow} \mathscr U_a(k)\]for every field $k$: the set on the left side consists of functors $F:T\rightarrow T_k$ sending $T^{U_a}$ to $0$, while the set on the right side consists of $F$ sending $a$ to zero.  But if $F(a)=0$ then $[F]\in U_a$, hence $T^{U_a}\subset \ker(F)$, and $F$ lies in the left side.  Therefore this map is in fact a bijection.

Taking colimits we then obtain:

\begin{cor}  For any $a\in T$, we have an isomorphism between locally ringed spaces \[|\mathrm{Sp}(T/T^{U_a})|\stackrel{\cong}{\longrightarrow} U_a \]induced by the localization functor $T\rightarrow T/T^{U_a}$, where $U_a=|\mathscr U_a|$ is a locally ringed subspace of $|\mathrm{Sp}(T)|$.
\end{cor}

\ppp{par:perfect}  Here we construct some triangulated tensor subcategories which will be used later in the proof of \rf{thm:main}.  For any tensor functor $F:T\rightarrow T_k$ denote by $\bar F: T/\ker(F)\rightarrow T_k$ the induced functor, and \[\bar F_1:\mathrm{End}_{T/\ker(F)}(1)\rightarrow \mathrm{End}_{T_k}(1)\cong k\] the evaluation of $\bar F$ on endomorphisms of $1$.

We denote by $k(F)$ the field \[\bar F(\mathrm{End}_{T/\ker(F)}(1))=\mathrm{End}_{T/\ker{(F)}}(1)/\ker(\bar F_1).\]and by $R(F)$ the local ring $\mathrm{End}_{T/\ker(F)}(1)$.  Notice that they depend only on $\ker(F)$, and in particular only on the class $[F]$.  The field $k(F)$ is contained in every field $k'$ so that $[F]$ is represented by a functor $T\rightarrow T_{k'}$.

Consider the functor \[F':a\mapsto \mathrm{Hom}_{T/\ker(F)}^\bullet(1,a)\otimes_{R(F)}k(F);\]here $\mathrm{Hom}^\bullet_{S}(x,y)$ denotes the direct sum \[\bigoplus_{j\in\mathbb Z}\mathrm{Hom}_S(x,y[j])\]for any triangulated category $S$.  Notice that this is a module over $\mathrm{End}_S(x)$.  

The functor $F'$ takes values in $T_{k(F)}$, since $F'(a)$ is naturally a graded $k(F)$-vector space; it is exact, but possibly not always a tensor functor.  Moreover, evaluation of $F$ at morphisms gives a natural transformation $\phi$ from $H:=F'\otimes_{k(F)}k$ to $F$:\[\phi_a: H(a)=\mathrm{Hom}_{T/\ker(F)}^\bullet(1,a)\otimes_{R(F)}k\longrightarrow \mathrm{Hom}_{T_k}^\bullet(F(1),F(a))\cong F(a).\]

We now have a diagram that is possibly not commutative:

\centerline{
\xymatrix{T\ar[r]^-{F'} \ar[dr]_-F& T_{k(F)}\ar[d]^-{-\otimes_{k(F)}k} \\ &T_k.}
}

Notice that $\phi_a: H(a)\rightarrow F(a)$ is an isomorphism whenever the image of $a$ in $T/\ker(F)$ is isomorphic to $1$.  Since both $H$ and $F$ are exact functors, the collection of objects $a\in T$ such that $\phi_a$ is an isomorphism is closed under shifting and taking cones.  

Denote by $P_{F}$ the strictly full triangulated subcategory of $T$ generated by objects $a\in T$ which are isomorphic to $1$ in $T/\ker(F)$; here by ``strictly'' we mean that if $a\in P_F$ and $a\cong b$ in $T$ then $b\in P_F$.  Equivalently, $P_F$ is the strictly full subcategory of $T$ consisting of objects whose images in $T/\ker(F)$ lie in the triangulated subcategory generated by $1\in T/\ker(F)$.  

Notice that $P_F$ is a strictly full \emph{triangulated tensor} subcategory of $T$ depending only on the class $[F]$, and we have a \emph{commutative} diagram

\begin{equation} \label{diag:PF} \xymatrix{P_F\ar[r]^-{F'} \ar[dr]_-F& T_{k(F)}\ar[d]^-{-\otimes_{k(F)}k} \\ &T_k.} \end{equation}

More precisely, the natural transformation $\phi$ is an isomorphism between the functors $H$ and $F$ when restricted to $P_F$.

Since $-\otimes_{k(F)}k$ is injective on objects, we see that $F'$ is an exact \emph{tensor} functor from $P_F$ to $T_{k(F)}$.  This gives the following

\begin{lemma}\label{lem:inj}  Let $F:T\rightarrow T_{k}$ and $G:T\rightarrow T_{K}$ be tensor functors with $\ker(F)=\ker(G)$; in particular $P_F=P_G=:P$.  Then $[F]$ and $[G]$ map to the same point under the natural map $|\mathrm{Sp}(T)|\rightarrow |\mathrm{Sp}(P)|$ induced by the inclusion $P\hookrightarrow T$.
\end{lemma}

\begin{proof}  Indeed, the condition $\ker(F)=\ker(G)$ implies moreover that \[F'=G': T\rightarrow T_{k(F)}=T_{k(G)}.\]

Diagram \rf{diag:PF} and the fact that $F'$ is a tensor functor shows that as points in $|\mathrm{Sp}(P)|$ we have $[F]=[F']=[G']=[G]$.\end{proof}


\mysec{Comparison morphisms}

\mysubsec{Prime spectrum}

\ppp{}  Balmer \cite{balmer_spectrumtt} defined for every triangulated tensor category $T$ a locally ringed space $\mathrm{Spec}(T)$ whose underlying space consists of prime tensor ideals in $T$.  This space is final among all topological spaces admitting a support data, which is a map from the objects in $T$ to the set of closed subsets of the topological space satisfying some natural axioms \cite[Definition 3.1]{balmer_spectrumtt}.  

By \rf{par:supportdata} we have a continuous map \[f: |\mathrm{Sp}(T)|\longrightarrow \mathrm{Spec}(T).\]  Explicitly, $f:[F]\mapsto \ker(F)$.  Moreover, a basis of open sets of $\mathrm{Spec}(T)$ are given by \[U(a):=\{\mathcal P\in\mathrm{Spec}(T)\,|\,a\in \mathcal P\},\]and we see that $f^{-1}(U(a))=U_a$.  The universal support data on $\mathrm{Spec}(T)$ is given by \[a\mapsto \mathrm{Spec}(T)-U(a).\]

\ppp{}  For any closed subset $Z$ of $\mathrm{Spec}(T)$, let\[T_Z:=\{a\in T\,|\, Z(a)\subset Z\};\]in other words, $T_Z$ consists of objects $a\in T$ such that $U(a)\supset U:=\mathrm{Spec}(T)-Z$.  From this we have \[T_Z=\bigcap_{\mathcal P\in U} \mathcal P.\]

For any open subset $U$ of $\mathrm{Spec}(T)$ we then have a natural inclusion \[T^{f^{-1}(U)} \supset T_Z,\]where $Z$ is the complement of $U$.  This induces a functor \[T/T_{Z}\longrightarrow T/T^{f^{-1}(U)},\]and in turn \[\mathrm{End}_{T/T_Z}(1)\longrightarrow f_*R(U)\] with notations as in \rf{par:R}.  Since the structure sheaf $\mathcal O_{\mathrm{Spec}(T)}$ is defined to be the sheafification of the presheaf given by the ring on the left side, we obtain a morphism $(f,f^\#)$ between ringed spaces.


\ppp{par:ex}  We remark that the comparison morphism $f: |\mathrm{Sp}(T)|\rightarrow \mathrm{Spec}(T)$ is not a bijection in general.  For a simple example, let $\ell$ be a field and $T$ be the orbit category $T_\ell/[m]$ whose objects are the same as those of $T_\ell$, but with \[\mathrm{Hom}_T(a,b)=\bigoplus_{j\in\mathbb Z}\mathrm{Hom}_{T_k}(a, b[jm]).\]  (Compositions are defined in the ``obvious'' way; see \cite[Definition 2.3]{CM_skewcat}.)

In \cite[Section 4, Theorem]{keller_orbitcat} and \cite[Lemma 2.3]{PX_rootcat} $T$ is proved to be triangulated so that the natural functor $T_\ell\rightarrow T$ (identity on objects) is an exact functor.  On the other hand, clearly the tensor product on $T_\ell$ gives a bifunctor $T\times T\rightarrow T$, and it is easy to show that this induced tensor product $\otimes_T$ on $T$ is exact in each variables:  It suffices to observe that the collection of objects $a\in T_\ell$ so that $a\otimes_T-:T\rightarrow T$ is an exact functor is closed under shifting and taking cones, and contains the unit object.

The set $|\mathrm{Sp}(T)|$ is empty:  In $T$ every object $a$ is isomorphic to $a[m]$, and in particular the unit object $\ell$ is isomorphic to its shift $\ell[m]$, hence there cannot be any exact tensor functor from $T$ to $T_k$ as long as $m\neq 0$.  Meanwhile, $\mathrm{Spec}(T)$ is always non-empty \cite[Proposition 2.3 (d)]{balmer_spectrumtt}.

\mysubsec{Scheme}

\ppp{}  For any scheme $X$ denote by $T_X$ the triangulated tensor category $D(X)_\mathrm{parf}$ \cite[3.1]{thomason_class}; its unit object is $\mathcal O_X$.  We define a morphism of locally ringed spaces \[\alpha: X\longrightarrow |\mathrm{Sp}(T_X)|\] as follows.

For any point $x\in X$, denote by $x^*:T_X\rightarrow T_{k(x)}$ the derived base change functor via $\mathrm{Spec}(k(x))\rightarrow X$, where $k(x)$ is the residue field of $\mathcal O_{X,x}$.  Then $\alpha:x\mapsto [x^*]$ defines $\alpha$ on the underlying sets.  

\ppp{par:topnoe}  Until the end of this section we will assume that $X$ is topologically noetherian; that is, its open sets satisfy the ascending chain condition.  Notice that this is equivalent to that all open sets are quasi-compact, and in particular implies that $X$ is quasi-compact and quasi-separated.

\ppp{}  To see that $\alpha$ is continuous, let $U_a\subset |\mathrm{Spec}(T_X)|$ be a basis open subset.  Then we have \[\alpha^{-1}(U_a)=\{x\in X\,|\,H^*(x^*a)=0\},\]namely, the complement on $X$ of the cohomology support.  By \cite[Lemma 3.4]{thomason_class} this is open, hence $\alpha$ is continuous.

\ppp{par:alphasharp}  Now we need to define $\alpha^\#_U:\mathcal O_{T_X}(U)\rightarrow \mathcal O_X(\alpha^{-1}(U))$ for any open set $U$ in $|\mathrm{Sp}(T_X)|$; for simplicity we denote $\alpha^{-1}(U)$ by $U'$.  For this it suffices to define on presheaves a homomorphism \[R(U)=\mathrm{End}_{T_X/T_X^U}(1)\rightarrow \mathcal O_X(U').\]

Notice that $\mathcal O_X(U')=\mathrm{End}_{T_{U'}}(\mathcal O_{U'})$ since $T_{U'}$ is a full subcategory of the derived category $D(U')=D(\mathcal O_{U'}\mathrm{-Mod})$ of the scheme $U'$.  Therefore it suffices to observe that the natural restriction functor $T_X\rightarrow T_{U'}$ sends objects in $T_X^U$ to zero, since then we have a factorization \begin{equation}\label{eq:U}T_X\longrightarrow T_X/T_X^U\longrightarrow T_{U'}\end{equation} of the restriction functor $j^*:T_X\rightarrow T_{U'}$.


\begin{lemma}\label{lem:ker} In the situation above we have $T_X^U\subset\ker(j^*:T_X\rightarrow T_{U'})$.
\end{lemma}

\begin{proof} $a\in T_X^U$ implies that for every $x\in U'$ we have that $x^*(a)$ is given by an acyclic complex in $T_{k(x)}$.  By \cite[Lemma 3.3 (a)]{thomason_class} this implies that $j^*(a)=a|_{U'}=0$.  
\end{proof}

\ppp{}  Let $x\in X$, then $x^*:T_X\rightarrow T_{k(x)}$ factors through \rf{eq:U} for any open set $U\subset |\mathrm{Sp}(T)|$ containing $[x^*]$ (in particular $U'$ contains $x$).  Hence if $r\in\mathcal O_{T,[x^*]}$ is represented by $\tilde r\in\mathrm{End}_{T/T^U}(1)$ (see \rf{par:stalk}) and mapped to zero under $\bar{x^*}:T_X/\ker(x^*)\rightarrow T_{k(x)}$, then the image of $\tilde r$ in $\mathrm{End}_{T_{U'}}(1)$ also maps to zero under $x^*: T_{U'}\rightarrow T_{k(x)}$.  Hence $(\alpha, \alpha^\#)$ is a morphism of locally ringed spaces. 



\mysubsec{Comparison theorem for schemes}

\ppp{}  Combining \rf{Prime spectrum} and \rf{Scheme} we have morphisms \[X\stackrel{\alpha}{\longrightarrow} |\mathrm{Sp}(T_X)|\stackrel{f}{\longrightarrow} \mathrm{Spec}(T_X)\] for any scheme $X$.  Denote by $\beta$ the composition $f\circ \alpha$.

In \cite[Corollary 5.6]{balmer_spectrumtt} it was proved that $\beta$ is an isomorphism of ringed spaces when $X$ is topologically noetherian.  We will prove under the same assumption that $\alpha$ is also an isomorphism.  

\begin{lemma}\label{lem:surj}  If $X$ is a topologically noetherian scheme, then for every $[F]\in |\mathrm{Sp}(T_X)|$ we have for a unique $x\in X$ such that $\ker(F)=\ker(x^*)$.
\end{lemma}
\begin{proof}  The map $\beta$ sends $x\mapsto \ker(x^*)$, and $f$ sends $[F]\mapsto \ker(F)$.\end{proof}



\ppp{}   The condition for objects in $P_F$ in \rf{par:perfect} mimics the notion of perfect complexes on a scheme.  Indeed, we have 

\begin{prop}\label{prop:PF}  Let $X$ be a topologically noetherian scheme; $T_X=D(X)_\mathrm{parf}$.  Then for every $[F]\in |\mathrm{Sp}(T_X)|$ we have $T_X=P_F$.
\end{prop}

\begin{proof}  Let $a$ be a complex representing an object in $T_X$, and $[F]\in |\mathrm{Sp}(T_X)|$.  By \rf{lem:surj} there is a point $x\in X$ such that $\ker(x^*)=\ker(F)$.  Hence to prove that $a$ lies in $P_F$ it suffices to show that $a$ lies in the strictly full triangulated subcategory generated by $1$ in $T_X/\ker(x^*)$.  

So let $M\subset T_X/\ker(x^*)$ be a strictly full triangulated subcategory containing $1$; we need to show that $a\in M$.

Choose an open neighbourhood $U'$ of $x$ in $X$ on which $a$ is quasi-isomorphic to a strict perfect complex; it exists by the definition of perfect complexes.  By shrinking $U'$ if necessary we may assume that $a$ is quasi-isomorphic on $U'$ to a bounded complex of finite direct sums of $1=\mathcal O_{U'}$.  By \cite[Lemma 3.4]{thomason_class} we can find an object $b$ in $T_X$ whose support on $X$ is equal to $X-U'$; in particular we have $\alpha^{-1}(U_b)=U'$.

With notations as in \rf{lem:ff} below, we see that $\bar j^*(a)$ lies in the triangulated subcategory $\bar j^*(M)$ containing $1=\mathcal O_{U'}\in T_{U'}/\ker(x^*)$.  But by \rf{lem:ff} and the strict fullness of $M$ we see that $(\bar j^{*})^{-1}(\bar j^*(M))=M$, hence $a$ lies in $M$.  \end{proof}

\begin{lemma}\label{lem:ff}  Let $X$ be a topologically noetherian scheme; $T_X=D(X)_\mathrm{parf}$.  For any open subset $U\subset |\mathrm{Sp}(T_X)|$, let $U'=\alpha^{-1}(U)\subset X$; let $x\in U'$.  Then the natural functor $\bar j^*:T_X/\ker(x^*)\longrightarrow T_{U'}/\ker(x^*)$ is fully faithful.  In particular, if $\bar j^*(a)$ is isomorphic to $\bar j^*(b)$ then $a$ is isomorphic to $b$ in $T_X/\ker(x^*)$.
\end{lemma}

\begin{proof}  We first prove that $\bar j^*$ is faithful.  Let $a,b$ be objects in $T_X/\ker(x^*)$, and $a\stackrel{s}{\leftarrow} c\stackrel{f}{\rightarrow} b$ be a morphism which is mapped to the zero morphism in $T_{U'}/\ker(x^*)$.  This means that we have a commutative diagram in $T_{U'}$

\centerline{\xymatrix{& d \ar[dr]^-0 \ar[d]^-u \ar[dl]_-t\\ j^*a & j^*c \ar[l]^-{j^*s} \ar[r]_-{j^*f}& j^*b}
}
where $x^*t$ is an isomorphism.  This implies that $x^*u$ is also an isomorphism, hence by applying \cite[Lemma 3.3 (a)]{thomason_class} to $x^*\mathrm{cone}(u)$ we see that there is some open neighbourhood $U''$ of $x$ in $U'$ such that $k^*u$ is already an isomorphism, where $k: U''\rightarrow U'$ is the inclusion.  

This implies that $k^*j^*f=0=k^*j^*0$ in $T_{U''}$.  By applying \cite[Proposition 5.2.4 (a)]{thomason_higherk} to the inclusion $U''\subset X$, we have $e\in T_X$ and morphisms $c\stackrel{w}{\leftarrow} e \stackrel{g}{\rightarrow} b$ with $x^*w$ an isomorphism and such that the following diagram is commutative:

\centerline{\xymatrix{& e \ar[dr]^-{f\circ w=0} \ar[d]^-w \ar[dl]_-{s\circ w}\\ a & c \ar[l]^-{s} \ar[r]_-{f}& b.}
}

This says precisely that the morphism $a\stackrel{s}{\leftarrow} c\stackrel{f}{\rightarrow} b$ we began with is zero.

Now we show that $\bar j^*$ is full.  Let $a,b$ be in $T_X$ and $j^*a\stackrel{s}{\leftarrow} c \stackrel{f}{\rightarrow }j^*b$ be an element in $\mathrm{Hom}_{T_{U'}/\ker(x^*)}(j^*a,j^*b)$; in particular $x^*s$ is an isomorhpism.  Then as above there is an open neighbourhood $U''$ of $x$ in $U'$ such that $k^*s$ is already an isomorphism.

Now $(k^*f)\circ (k^*s)^{-1}: k^*j^*a\rightarrow k^*j^*b$ is a morphism in $T_{U''}$.  By \cite[Proposition 5.2.3 (a)]{thomason_higherk} there is an object $d\in T_X$ along with morphisms $a\stackrel{t}{\leftarrow} d \stackrel{g}{\rightarrow} b$ such that $k^*j^*t$ is an isomorphism in $T_{U''}$ and the following diagram is commutative:

\centerline{\xymatrix{ & k^*j^*d \ar[dl]_-{k^*j^*t} \ar[dr]^-{k^*j^*g}\\ k^*j^* a \ar[r]_-{(k^*s)^{-1}}& k^*c\ar[r]_-{k^*f}  & k^*j^*b.}
}

We will show that the image of $a\stackrel{t}{\leftarrow} d \stackrel{g}{\rightarrow} b$ under $\bar j^*$ is equal to $j^*a\stackrel{s}{\leftarrow} c \stackrel{f}{\rightarrow }j^*b$.  

By \cite[Proposition 5.2.3 (a)]{thomason_higherk} applied to the morphism $(k^*s)^{-1}\circ(k^*j^*t):k^*j^*d\rightarrow k^*c$ above, there is an object $e$ in $T_{U'}$ along with morphisms $j^*d\stackrel{r}{\leftarrow} e \stackrel{h}{\rightarrow} c$ so that $k^*r$ is an isomorphism and the following diagrams are commutative:

\centerline{\xymatrix{ & k^*j^*d \ar[dl]_-{k^*j^*t}   & k^*e\ar[l]_-{k^*w} \ar[dl]^-{k^*h}  \\ k^*j^* a & k^*c\ar[l]^-{k^*s},}
}
\bigskip
\centerline{\xymatrix{ k^*e \ar[r]^-{k^*w}  \ar[dr]_-{k^*h}& k^*j^*d  \ar[dr]^-{k^*j^*g}\\ & k^*c\ar[r]_-{k^*f}  & k^*j^*b.}
}

Now by the faithfulness of $\bar k^*:T_{U'}/\ker(x^*)\rightarrow T_{U''}/\ker(x^*)$ proved above, we see that there is another object $e'$ in $T_{U'}$ along with morphisms $e'\rightarrow e$, $e'\rightarrow j^*d$, and $e'\rightarrow c$, with the first morphism becoming an isomorphism under $x^*$, and such that the following diagram is commutative:

\centerline{\xymatrix{ & j^*d \ar[dl]_-{j^*t} \ar[dr]^-{j^*g}\\ j^* a&  e' \ar[d]\ar[u]\ar[l]\ar[r] & j^*b \\ & c\ar[ur]_-{f}  \ar[ul]^-{s}, }
}
where the arrows to the left and right are defined to be the obvious compositions.  This diagram says that the image of $a\stackrel{t}{\leftarrow} d \stackrel{g}{\rightarrow} b$ under $\bar j^*$ is equal to $j^*a\stackrel{s}{\leftarrow} c \stackrel{f}{\rightarrow }j^*b$, and so $\bar j^*$ is full.
\end{proof}

\ppp{}  Returning to the comparison morphisms \[X\stackrel{\alpha}{\longrightarrow} |\mathrm{Sp}(T_X)|\stackrel{f}{\longrightarrow} \mathrm{Spec}(T_X),\] we now have

\begin{prop}\label{prop:homeo}  The map $\alpha$ is a homeomorphism when $X$ is a topologically noetherian scheme.
\end{prop}

\begin{proof}  By \cite[Corollary 5.6]{balmer_spectrumtt}, the composition $\beta=f\circ\alpha$ is a homeomorphism.  Hence it suffices to show that the continuous map $\alpha$ is a bijection, and it only remains to show that it is surjective.

Let $[F]\in|\mathrm{Sp}(T_X)|$.  By \rf{lem:surj} there is a unique $x\in X$ such that $\ker(F)=\ker(x^*)$.  By \rf{prop:PF} we have $T_X=P_F=P_{x^*}$, and by \rf{lem:inj} we have $[F]=[x^*]$.\end{proof}

\ppp{}  Now we can finish the proof of

\begin{thm}\label{thm:main}  If $X$ is a topologically noetherian scheme and $T_X=D(X)_\mathrm{parf}$, then the comparison morphism \[X\stackrel{\alpha}{\longrightarrow} |\mathrm{Sp}(T_X)|\] is an isomorphism.
\end{thm}

\begin{proof}  By \rf{prop:homeo} we know that $\alpha$ is a homeomorphism, so it only remains to show that $\alpha^\#$ as defined in \rf{par:alphasharp} is a sheaf isomorphism.  For this it suffices to show that for every open set $U\subset |\mathrm{Sp}(T_X)|$ the homomorphism \[\mathrm{End}_{T_X/T_X^U}(1)\longrightarrow \mathrm{End}_{T_{U'}}(1)\] induced by the restriction functor $j^*:T_X/T_X^U\rightarrow T_{U'}$ is an isomorphism; note that $U'=\alpha^{-1}(U)$ is now identified with $U$.  But by \cite[Proposition 5.2.3 (a), Proposition 5.2.4 (a)]{thomason_higherk} the functor $j^*$ is fully faithful.  \end{proof}

\mysubsec{Comparison theorem for quivers}

\ppp{}  Let $Q$ be a quiver possible with relations.  We require that the relations satisfy the property that the category $Q\mathrm{-Rep}_\ell$ of finite dimensional $Q$-representations over a field $\ell$ is monoidal under the usual vertex-wise tensor product.  Note that this tensor product is exact, and $T_{Q,\ell}:=D^b(Q\mathrm{-Rep}_\ell)$ is then a $\ell$-linear triangulated tensor category.

\ppp{}  Let $Q_0$ b the set of vertices of $Q$.  Then for any $v\in Q_0$ we have a tensor functor $F_v: T_{Q,\ell}\rightarrow T_\ell$ sending \[V\mapsto H^*(V)_v,\]where the cohomology $H^*(V)$ of an object $V\in T_{Q,\ell}$ is a complex with zero differential with objects in $Q\mathrm{-Rep}_\ell$.  The complex of $\ell$-vector spaces of $H^*(V)_v$ at the vertex $v$ is then naturally an object in $T_\ell$.  

This defines a map from $Q_0$ to $\mathrm{Sp}(T_{Q,\ell})(k)$ for every field extension $\ell\rightarrow k$, and hence a map from $Q_0$ to $|\mathrm{Sp}(T_{Q,\ell})|$.

The following is proved in \cite{liu_sierra_quiver}:

\begin{prop}  If $Q$ is finite and without oriented cycles, then for every field we have homeomorphisms \[Q_0\longrightarrow |\mathrm{Sp}(T_{Q,\ell})|\longrightarrow \mathrm{Spec}(T_{Q,\ell}),\]where $Q_0$ is given the discrete topology.  These maps are given by $v\mapsto [F_v]$ and $[F]\mapsto \ker(F)$. 
\end{prop}

\ppp{}  Associated to the quiver $Q$ we have an algebra-valued functor on the category of fields sending \[k\mapsto A(Q,k):=kQ,\] the path algebra of $Q$.  Fixing a field $\ell$ and restricting this to field extensions $\ell\rightarrow k$, we will define a natural transformation between functors \[\phi: A(Q,-)\longrightarrow A(T_{Q,\ell},-),\]where the second functor was defined in \rf{par:path}.  

Explicitly, if $p$ is a path from vertex $v$ to vertex $w$ giving an element $p\in A(Q,k)$, then for any field extension $\ell\rightarrow k$, $\phi_k(p)$ is the element in $\mathrm{Hom}(F_v\otimes_\ell k, F_w\otimes_\ell k)\subset A(T_{Q,\ell},k)$ sending \[V\mapsto \phi_k(p)_V:=(V_p: F_v(V)\rightarrow F_w(V)).\]

The following is proved in \cite{liu_sierra_quiver}:

\begin{thm}\label{thm:quiver}  If $Q$ is finite and without oriented cycles, then for every field $\ell$ the natural transformation $\phi$ above is an isomorphism from $A(Q,-)$ to $A(T_{Q,\ell},-)$.
\end{thm}

\mysec{Rigidifying tensor products}

\mysubsec{Identifying tensor products}

Here we prove a few simple preliminary lemmas that will help us compare tensor products on triangulated categories.

\ppp{}  We consider first the situation when there is an exact equivalence $g: T\rightarrow T'$ of triangulated categories, when $T'$ has a given tensor product $\otimes'$.  Denote by $g'$ a quasi-inverse to $g$, then we can define a tensor product $\otimes_1$ on $T$ by setting \[a\otimes_1b:=g'(g(a)\otimes'g(b)).\]

\begin{lemma}\label{lem:t1}  The functor $g: (T,\otimes_1)\rightarrow (T',\otimes')$ is an equivalence of triangulated tensor categories.  In particular we have an isomorphism $|\mathrm{Sp}(T',\otimes')|\cong |\mathrm{Sp}(T,\otimes_1)|$ of locally ringed spaces.
\end{lemma}

\begin{proof}  It only remains to show that $g$ is a strong monoidal functor:\[g(a\otimes_1b)=gg'(g(a)\otimes'g(b))\cong g(a)\otimes'g(b).\]\end{proof}

\ppp{}  Let $u$ be an object in a triangulated tensor category $(T,\otimes_1)$ be such that the exact functor $m: T\rightarrow T$ defined by $a\mapsto u\otimes_1a$ is an equivalence with quasi-inverse $m'$.  Then we can define a tensor product $\otimes_0$ on $T$ by setting \[a\otimes_0b:=m(m'(a)\otimes_1 m'(b)).\]

\begin{lemma}\label{lem:t2}  The functor $m: (T,\otimes_1)\rightarrow (T, \otimes_0)$ is an equivalence of triangulated tensor categories.  In particular we have an isomorphism $|\mathrm{Sp}(T,\otimes_0)|\cong |\mathrm{Sp}(T,\otimes_1)|$ of locally ringed spaces.

\end{lemma}

Note that since the unit object $1$ with respect to $\otimes_1$ is isomorphic to $g'(u)$, the object $u$ is the unit object with respect to $\otimes_0$.

\begin{proof}  We compute: \[m(a\otimes_1 b)\cong m(m'm(a)\otimes_1 m'm(b))=m(a)\otimes_0m(b).\]  \end{proof}

\ppp{}  Let $T$ be a triangulated category, and $W\subset T$ a full subcategory, not necessarily triangulated.  We say that \emph{$W$ generates $T$} if for every $a\in T$ there are finitely many objects $w_1, w_2, \ldots, w_m$ in $W$ and distinguished triangles in $T$:
\[a_{1}\longrightarrow w_1\longrightarrow a\longrightarrow a_1[1],\]
\[a_2\longrightarrow w_2\longrightarrow a_1\longrightarrow a_2[1],\]
\[\ldots,\]
\[w_m\longrightarrow w_{m-1}\longrightarrow a_{m-1}\longrightarrow w_m[1],\]for some objects $a_1, \ldots, a_{m-1}$.

\begin{lemma}\label{lem:t3}  If $(T,\otimes)$ is a triangulated tensor category, then $\otimes$ is determined by its restriction to $W\times T\rightarrow T$ for any full subcategory $W$ generating $T$.
\end{lemma}

\begin{proof}  More precisely, if $\otimes$ and $\otimes_0$ are tensor products on $T$ whose restrictions to $W\times T$ are isomorphic, then the isomorphism extends to an isomorphism between $\otimes$ and $\otimes_0$ on $T\times T$.

Indeed, suppose $\phi_{w,b}: w\otimes b\stackrel{\cong}{\longrightarrow} w\otimes_0 b$, $w\in W$, $a\in T$, is a natural isomorphism between the restrictions, we need to extend to an isomorphism $\phi_{a,b}: a\otimes b\stackrel{\cong}{\longrightarrow} a\otimes_0 b$, $a,b\in T$.  To this end we only need to find a ``resolution'' of $a$ by objects in $W$ as above and apply $\phi_{-,b}: -\otimes b\rightarrow -\otimes_0 b$ to every distinguished triangle.  \end{proof}

\mysubsec{Example: smooth varieties with ample (anti-)canonical bundles}

\ppp{}  We will give a proof of the following 

\begin{thm}\label{thm:BO} \cite[Theorem 2.5]{BO_reconstruction}  Let $X$ be an irreducible smooth projective variety over a field $k$ with ample (anti-)canonical bundle.  If $X'$ is another irreducible smooth projective variety over $k$ such that $D^b(X)$ and $D^b(X')$ are equivalent as triangulated categories, then $X$ is isomorphic to $X'$.  \end{thm}

Our proof is not completely independent of that in \cite{BO_reconstruction}; see \rf{lem:points} below.

Under the smoothness assumption, $D^b(X)$ is equivalent to $T_X=D(X)_\mathrm{parf}$, therefore the knowledge of the whole tensor product $T_X\times T_X\rightarrow T_X$ already allows us to reconstruct $X$ \rf{thm:main}.  

The main point of the proof is that the ampleness assumption allows us to recover the whole tensor product from the restricted tensor product $W\times T_X\rightarrow T_X$ for certain subcategory $W$ generating $T_X$, which in turn is determined by the unique Serre functor.  

\ppp{}  Denote by $g: T_X\rightarrow T_{X'}$ an exact equivalence with quasi-inverse $g'$; denote by $\otimes$ and $\otimes'$ their respective tensor products.

\begin{lemma}\label{lem:points}  With assumptions as in \rf{thm:BO}, $g(\mathcal O_X)$ is isomorphic to a line bundle on ${X'}$.
\end{lemma}

\begin{proof}  For any closed point $x'\in X'$, by \cite[Proposition 2.2]{BO_reconstruction} we have that $g'(\mathcal O_{x'})\cong \mathcal O_x[i]$ for some integer $i\in\mathbb Z$ and closed point $x\in X$, depending on $x'$.  Note that the residue field $k(x)$ can be recovered as $\mathrm{End}_{T_X}(\mathcal O_x)$, in particular we have $k(x)\cong k(x')$; denote this field by $k'$.

Hence we have \[\mathrm{Hom}_{T_{k'}}(x'^*g(\mathcal O_X),k'[j])\cong \mathrm{Hom}_{T_{X'}}(g(\mathcal O_X),\mathcal O_{x'}[j])\cong \mathrm{Hom}_{T_X}(\mathcal O_X,\mathcal O_{x}[i+j]),\]which is isomorphic to $k'$ when $i+j=0$, and is equal to zero when $i+j\neq 0$. 

In particular $x'^*g(\mathcal O_X)$ is non-zero for every closed point $x'\in X'$, hence the support of $g(\mathcal O_X)$ is the whole of $X'$.  Moreover, representing $g(\mathcal O_X)$ by a perfect complex on $X'$ we see that this complex has only one non-zero cohomology sheaf (in particular the integer $i$ is independent of $x'$), which must be a line bundle on $X'$ since all of its fibres have rank one.  \end{proof}

\ppp{}  Now we can give a

\begin{proof}[Proof of Theorem \rf{thm:BO}]  Denote by $g: T_X\rightarrow T_{X'}$ an exact equivalence with quasi-inverse $g'$; denote by $\otimes$ and $\otimes'$ the respective derived tensor products on these two triangulated categories.  By \rf{lem:t1} we have a tensor product $\otimes_1$ on $T_X$ defined by \[a\otimes_1b=g'(g(a)\otimes' g(b))\] so that there are isomorphisms of locally ringed spaces \[X'\stackrel{\cong}{\longrightarrow}|\mathrm{Sp}(T_{X'},\otimes')|\stackrel{\cong}{\longrightarrow} |\mathrm{Sp}(T_X,\otimes_1)|.\]

Denote by $\omega$ and $\omega'$ the canonical line bundles on $X$ and $X'$ respectively, and by $S,S'$ the \emph{shifts} of the Serre functors so that $S(a)=a\otimes \omega$ and $S'(a')=a'\otimes'\omega'$; we will use the fact that Serre functors commute with exact equivalences \cite[Proposition 1.3]{BO_reconstruction}.  Let $\omega_1:=g'(\omega')$, then we have natural isomorphisms \[a\otimes_1\omega_1\cong g'(g(a)\otimes' \omega')=g'S'g(a)\cong g'gS(a)\cong a\otimes \omega,\]for every $a\in T_X$.  In particular, by taking $a=\mathcal O_X$ we have $\mathcal O_X\otimes_1\omega_1\cong \omega$.  

Now by \rf{lem:points} we know that $g(\mathcal O_X)$ is a line bundle on $X'$, hence in particular the functor $a'\mapsto g(\mathcal O_X)\otimes'a'$ is an autoequivalence on $T_{X'}$.  It then follows that $m:a\mapsto \mathcal O_X\otimes_1a$ is an autoequivalence on $T_X$; denote by $m'$ its quasi-inverse.  Hence by \rf{lem:t2} there is another tensor product $\otimes_0$ on $T_X$ defined by \[a\otimes_0b=m(m'(a)\otimes_1m'(b))\] so that we now have isomorphisms \[X'\stackrel{\cong}{\longrightarrow} |\mathrm{Sp}(T_X,\otimes_1)|\stackrel{\cong}{\longleftarrow} |\mathrm{Sp}(T_X,\otimes_0)|.\]

Now we compute: \[a\otimes_0\omega \cong m(m'(a)\otimes_1m'(\mathcal O_X\otimes_1\omega_1))\cong m(m'(a)\otimes_1 \omega_1)\cong mSm'(a)\cong mm'S(a)\cong S(a).\] Hence we have natural isomorphisms \[(a\mapsto S(a))\cong(a\mapsto a\otimes_0\omega)\cong (a\mapsto a\otimes \omega).\]

Let $W\subset T_X$ be the full subcategory consisting of objects of the form $\displaystyle\bigoplus_j\omega^{\otimes \ell_j}[m_j]$ with $\ell_j, m_j\in \mathbb Z$ and only finitely many summands.  Then since either $\omega$ or $\omega^{-1}$ is ample, $W$ generates $T_X$ \cite[Proposition 2.3.1 (d)]{thomason_higherk}.  But the restrictions of $\otimes$ and $\otimes_0$ to $W\times T_X\rightarrow T_X$ both give direct sums of iterations of shifts of the Serre functor, and so by \rf{lem:t3} we know that $\otimes$ and $\otimes_0$ are isomorphic, and we conclude:\[X'\stackrel{\cong}{\longrightarrow} |\mathrm{Sp}(T_X,\otimes_0)|\cong |\mathrm{Sp}(T_X,\otimes)|\stackrel{\cong}{\longleftarrow} X.\]    \end{proof}

\bibliographystyle{plain}
\bibliography{mybibli}

 \end{document}